\documentclass[twoside,leqno, british,12pt]
{amsart}
\usepackage{amsmath, amssymb, amsbsy, amsthm}
\usepackage{eucal,url,amssymb,
enumerate,amscd,stmaryrd
}

\usepackage[T1]{fontenc}

\usepackage[british]{babel}
 \newtheorem{thm}{Theorem}[section]
 \newtheorem{cor}[thm]{Corollary}
 \newtheorem{lem}[thm]{Lemma}
 \newtheorem{prop}[thm]{Proposition}
 \theoremstyle{definition}
 
 \theoremstyle{remark}

 \numberwithin{equation}{section}

\begin{document}
%
%
%
%
%
%
%
%
%
\title[Slant Null Curves on Some 3-Manifolds]
 {Slant Null Curves on Normal Almost 
 Contact B-Metric 3-Manifolds with parallel Reeb vector field}

\author[G. Nakova]{Galia Nakova}

\address{%
University of Veliko Tarnovo "St. Cyril and St. Methodius" \\ Faculty of Mathematics and Informatics\\   Department of Algebra and Geometry
\\ 2 Teodosii Tarnovski Str. \\ Veliko Tarnovo 5003\\ Bulgaria}
\email{gnakova@gmail.com}

\author[H. Manev]{Hristo Manev$^{1,2}$}
\address[1]{Medical University of Plovdiv\\ Faculty of Pharmacy\\
Section of Mathematics and IT\\ 15-A Vasil Aprilov Blvd.\\   Plovdiv 4002\\   Bulgaria}
\address[2]{Paisii Hilendarski University of Plovdiv\\   Faculty of Mathematics and
Informatics\\   Department of Algebra and Geometry\\   236
Bulgaria Blvd.\\   Plovdiv 4027\\   Bulgaria}
\email{hmanev@uni-plovdiv.bg}

\thanks{The second author was partially supported by projects NI15-FMI-004 and  MU15-FMIIT-008
of the Scientific Research Fund at the University of Plovdiv.}




\begin{abstract}
In this paper we study slant null curves with respect to the original parameter on 3-dimensional normal almost contact B-metric manifolds with parallel
Reeb vector field. We prove that for non-geodesic such curves there exists a unique Frenet frame for which the original parameter is distinguished. Moreover, we
obtain a necessary condition this Frenet frame to be a Cartan Frenet frame with respect to the original parameter. Examples of the considered curves are constructed.
\end{abstract}

\newcommand{\ie}{i.\,e. }
\newcommand{\g}{\mathfrak{g}}
\newcommand{\D}{\mathcal{D}}
\newcommand{\F}{\mathcal{F}}
\newcommand{\diag}{\mathrm{diag}}
\newcommand{\End}{\mathrm{End}}
\newcommand{\im}{\mathrm{Im}}
\newcommand{\id}{\mathrm{id}}
\newcommand{\Hom}{\mathrm{Hom}}

\newcommand{\Rad}{\mathrm{Rad}}
\newcommand{\rank}{\mathrm{rank}}
\newcommand{\const}{\mathrm{const}}
\newcommand{\tr}{{\rm tr}}
\newcommand{\ltr}{\mathrm{ltr}}
\newcommand{\codim}{\mathrm{codim}}
\newcommand{\Ker}{\mathrm{Ker}}
\newcommand{\R}{\mathbb{R}}

\newcommand{\thmref}[1]{Theorem~\ref{#1}}
\newcommand{\propref}[1]{Proposition~\ref{#1}}
\newcommand{\corref}[1]{Corollary~\ref{#1}}
\newcommand{\secref}[1]{\S\ref{#1}}
\newcommand{\lemref}[1]{Lemma~\ref{#1}}
\newcommand{\dfnref}[1]{Definition~\ref{#1}}


\newcommand{\ee}{\end{equation}}
\newcommand{\be}[1]{\begin{equation}\label{#1}}

\maketitle

\section{Introduction}\label{sec-1}
Many of the results in the classical differential geometry of curves have analogues in the Lorentzian geometry.  Since null curves have very different properties compared to
space-like and time-like curves, we have special interest in studying the geometry of null curves. The general theory of null curves is developed in \cite{D-B, D-J},
where there are established important applications of these curves in general relativity.

Let $\bf{F}$ be a Frenet frame along a null curve $C$ on a Lorentzian manifold. According to \cite{D-B}, $\bf{F}$ and the Frenet equations with respect to $\bf{F}$ depend on both the parametrization of $C$ and the choice of a screen vector bundle. However, if a non-geodesic
null curve $C$ is  properly parameterized, then there exists only one Frenet frame, called a Cartan Frenet frame, for which
the corresponding Frenet equations of $C$, called Cartan Frenet equations, have minimum number of curvature functions (\cite{D-J}).

In this paper we consider 3-dimensional almost contact B-metric manifolds $(M,\varphi,\allowbreak{}\xi,\eta,g)$, which are Lorentzian manifolds equipped with an almost contact B-metric structure. We study slant null curves on considered manifolds belonging to the class $\F_1$ of the Ganchev-Mihova-Gribachev classification given in \cite{GaMGri}. Also, the special class $\F_0$, which is subclass of $\F_1$, is object of considerations.

A slant curve $C(t)$ on $(M,\varphi,\xi,\eta,g)$, defined by the
condition $g(\dot C(t),\xi )={\rm const}$ for the tangent vector $\dot C(t)$, is a natural generalization of a cylindrical helix in an Euclidean space. Let us remark that if we change the parameter $t$ of a slant curve $C(t)$ with another parameter $p$, then we have $\dot C(p)=\dot C(t)\frac{{\rm d}t}{{\rm d}p}$. Hence
$g(\dot C(p),\xi )$ is a
constant if and only if $t=ap+b$, where $a, \, b$ are real numbers. This means that a slant null curve $C(t)$ is not slant with respect to a special parameter $p$ in
general. Motivated by this fact, our aim in the present work is to study slant null curves with respect to the original parameter. In \cite{W} there are considered non-geodesic slant null curves with respect to a special parameter $p$ on 3-dimensional normal almost paracontact metric manifolds.

The paper is organized as follows.
Section 2 is a brief review of almost contact B-metric manifolds and geometry of null curves on a 3-dimensional Lorentzian manifold.
The main results are presented
in Section 3. First, in \propref{Proposition 3.1} we express a general Frenet frame ${\bf F}$ and the functions $h$ and $k_1$ along a slant null curve $C(t)$ on $(M,\varphi,\xi,\eta,g)$ in terms of the almost contact B-metric structure. In \thmref{Theorem 3.1} we prove that for a non-geodesic slant null curve $C(t)$ on a 3-dimensional $\F_1$-manifold there exists a unique Frenet frame ${\bf F}_1$ for which the original parameter $t$ is distinguished. Let us remark that in \cite{H-I} it is shown that for any non-geodesic null curve on a 3-dimensional Minkowski space there exists a unique Frenet frame for which the original parameter is distinguished. In \thmref{Theorem 3.2} we give a necessary
condition for a slant null curve $\left(C(t), {\bf F}_1\right)$ to be a non-geodesic Cartan framed null curve with respect to the original parameter $t$.
The last Section 4 is devoted to some examples of the investigated curves.

\section{Preliminaries}\label{sec-2}

Let $(M,\varphi,\xi,\eta,g)$ be an almost contact manifold with B-metric or an {\it almost contact B-metric manifold}. This means that $M$ is a $(2n+1)$-dimensional differentiable manifold, $(\varphi,\xi,\eta)$ is an almost contact structure consisting of an endomorphism $\varphi $ of the tangent bundle, a Reeb vector field $\xi $ and its dual contact 1-form $\eta $, \ie the following relations are satisfied (\cite{GaMGri}):
\begin{equation*}
\varphi^2X=-\id+\eta\otimes \xi, \qquad \quad \eta(\xi)=1,
\end{equation*}
where $\id$ denotes the identity; as well as $g$ is a pseudo-Riemann\-ian metric, called a \emph{B-metric},
such that (\cite{GaMGri})
\begin{equation*}
g(\varphi X,\varphi Y)=-g(X,Y)+\eta(X)\eta(Y).
\end{equation*}
Here and further $X$, $Y$ and $Z$ are tangent vector fields on $M$, \ie $X, Y \in \Gamma (TM)$.
Immediate consequences of the above conditions are:
\begin{equation}\label{1'}
\eta \circ \varphi =0, \quad \varphi \xi =0, \quad {\rm rank}(\varphi)=2n, \quad \eta (X)=
g(X,\xi ), \quad g(\xi,\xi )=1.
\end{equation}
The tensor ${\widetilde g}$ given by ${\widetilde g}(X,Y)=g(X,\varphi Y)+\eta (X)\eta (Y)$ is a B-metric, too. Both metrics $g$ and
${\widetilde g}$ are of signature $(n+1,n)$.

Let $\nabla$ be the Levi-Civita connection of $g$. The tensor field $F$ of type $(0,3)$ on $M$ is defined by
$F(X,Y,Z)=g((\nabla_X\varphi)Y,Z)$ 
and it has the following properties:
\[
F(X,Y,Z)=F(X,Z,Y)=F(X,\varphi Y,\varphi Z)+\eta (Y)F(X,\xi,Z)+\eta (Z)F(X,Y,\xi ).
\]
Moreover, we have
\begin{equation}\label{2.1}
F(X,\varphi Y,\xi )=(\nabla _X\eta )Y=g(\nabla _X\xi,Y).
\end{equation}
The following 1-forms, called \emph{Lee forms}, are associated with $F$:
\[
\theta (X)=g^{ij}F(e_i,e_j,X), \quad \theta ^*(X)=g^{ij}F(e_i,\varphi e_j,X), \quad
\omega (X)=F(\xi,\xi,X),
\]
where $\{e_i,\xi \}, \, i=\{1,\ldots,2n\}$ is a basis of $T_xM$, $x\in M$, and $(g^{ij})$ is the inverse matrix of $(g_{ij})$.

A classification of the almost contact B-metric manifolds with respect to $F$ is given in \cite{GaMGri} and
eleven basic classes  $\F_i$ $(i=1,2,\dots,11)$ are obtained. If $(M,\varphi,\xi,\eta,g)$ belongs to $\F_i$ then it is called
an \emph{$\F_i$-manifold}.

The special class $\F_0$ is the intersection of all basic classes. It is known as the class of the {\it cosymplectic B-metric manifolds}, \ie the class of the considered manifolds with parallel structure tensors with respect to $\nabla$, namely
$\nabla \varphi =\nabla \xi =\nabla \eta =\nabla g=\nabla {\widetilde g}=0$ and consequently $F=0$.

The lowest dimension of the considered manifolds is dimension three. In \cite{HM1} it is established that the class of 3-dimensional almost contact B-metric manifolds is
$\F_1\oplus \F_4\oplus \F_5\oplus \F_8\oplus \F_9\oplus \F_{10}\oplus \F_{11}$.
According to \cite{ManIv13}, the class of the normal almost contact B-metric manifolds is $\F_1\oplus\F_2\oplus\F_4\oplus\F_5\oplus\F_6$, since the Nijenhuis tensor of almost contact structure vanishes there. Therefore, the class of the 3-dimensional normal almost contact B-metric manifolds is $\F_1\oplus\F_4\oplus\F_5$.

In this paper we consider 3-dimensional almost contact B-metric manifolds belonging to $\F_1$, determined by (\cite{GaMGri})
\begin{equation}\label{2.2}
\begin{array}{ll}
\F_1 : F(X,Y,Z)=\frac{1}{2}\left\{g(X,\varphi Y)\theta (\varphi Z)+g(\varphi X,\varphi Y)\theta (\varphi ^2Z)\right. \\
\phantom{\F_1 : F(X,Y,Z)=\frac{1}{2}\ }
\left. +g(X,\varphi Z)\theta (\varphi Y)+g(\varphi X,\varphi Z)\theta (\varphi ^2Y)\right\}.
\end{array}
\end{equation}

Taking into account \eqref{2.1} and \eqref{2.2}, for $\F_1$-manifolds we have that
\begin{equation*}\label{2.2'}
  \nabla\xi=0 .
\end{equation*}
Let us remark that $\nabla\xi\neq 0$ for the rest 3-dimensional normal almost contact B-metric manifolds.
Therefore, we can determine $\F_1$ as the class of 3-dimension\-al normal almost contact B-metric manifolds with parallel Reeb vector field.
Obviously, $\F_0$-manifolds are $\F_1$-manifolds with vanishing Lee forms, \ie $\F_0$ is the subclass of $\F_1$ of the so-called \emph{balanced manifolds} of the considered type.

Let us remark that on a 3-dimensional $(M,\varphi,\xi,\eta,g)$ the metric $g$ has signature $(2,1)$, \ie $(M,g)$ is a 3-dimensional Lorentzian manifold.

Let 
$C: I\longrightarrow M$ be a smooth curve on $M$ given locally by
\[
x_i=x_i(t), \quad t\in I\subseteq {\R}, \quad i\in \{1,2,3\}
\]
for a coordinate neighborhood $U$ of $C$. The tangent vector field is given by
\[
\frac{{\rm d}}{{\rm d}t}=(\dot {x}_1, \dot {x}_2, \dot {x}_3)=\dot {C},
\]
where we denote $\frac{{\rm d}x_i}{{\rm d}t}$ by $\dot {x}_i$ for $i\in \{1,2,3\}$. The curve $C$ is called a {\it regular curve} if $\dot {C}\neq 0$ holds everywhere.

Let a regular curve $C$ be a null (lightlike) curve on $(M, g)$, \ie at each point $x$ of $C$ we have
\begin{equation}\label{2'}
g(\dot {C},\dot {C})=0,\qquad \dot {C}\neq 0.
\end{equation}
%
A general Frenet frame on $M$ along $C$ is denoted by ${\bf F}=\{\dot {C}, N, W\}$ and determined by
\begin{equation}\label{3'}
g(\dot {C},N)=g(W,W)=1, \quad g(N,N)=g(N,W)=g(\dot {C},W)=0.
\end{equation}
The following general Frenet equations with respect to ${\bf F}$ and $\nabla $ of $(M, g)$ are known from \cite{D-J} %
\begin{equation}\label{general Frenet eq}
\begin{array}{lll}
\nabla _{\dot {C}}\dot {C}=h\dot {C}+k_1W, \\
\nabla _{\dot {C}}N=-hN+k_2W, \\
\nabla _{\dot {C}}W=-k_2\dot {C}-k_1N,
\end{array}
\end{equation}
where 
$h$, $k_1$ and  $k_2$
are smooth functions on $U$. The functions $k_1$ and  $k_2$ are called {\it curvature functions} of $C$.

The general Frenet frame ${\bf F}$ and its general Frenet equations \eqref{general Frenet eq} are not unique as they depend on the parameter and the choice of the
screen vector bundle of $C$ (for details see \cite[pp. 56-58]{D-B}, \cite[pp. 25-29]{D-J}). It is known \cite[p. 58]{D-B} that there exists a
parameter $p$ called a {\it distinguished parameter}, for which the function $h$ vanishes in \eqref{general Frenet eq}. The pair $(C(p), {\bf F})$, where ${\bf F}$ is a Frenet frame along $C$ with respect to a distinguished parameter $p$, is called a {\it framed null curve} (see \cite{D-J}). In general, $(C(p), {\bf F})$ is not unique since it depends on both $p$ and the screen distribution. Therefore we look for a Frenet frame with the minimum number of curvature functions
which are invariant under Lorentzian transformations. Such frame is called {\it Cartan Frenet frame} of a null curve $C$. In \cite{D-J} it is proved that if the null curve $C(p)$ is non-geodesic
such that the following condition for $\ddot{C}=\frac{{\rm d}}{{\rm d}p}\dot{C}$ holds
\begin{equation*}\label{k1}
g(\ddot {C},\ddot {C})=k_1=1,
\end{equation*}
then there exists only one Cartan Frenet frame
${\bf F}$ 
with the following Frenet equations
\begin{equation}\label{Cartan Frenet eq}
\begin{array}{lll}
\nabla _{\dot {C}}\dot {C}=W, \\
\nabla _{\dot {C}}N=\tau W, \\
\nabla _{\dot {C}}W=-\tau \dot {C}-N.
\end{array}
\end{equation}
The latter equations are called the {\it Cartan Frenet equations} of $C(p)$ whereas $\tau $ is called a \emph{torsion function} and it is invariant upto a sign under Lorentzian transformations. A null curve together with its Cartan Frenet frame is called a {\it Cartan framed null curve}.

Let us consider a smooth curve $C$ 
on an almost contact B-metric manifold $(M,\varphi,\xi,\allowbreak{}\eta,g)$. We say
that $C$ is a {\it slant curve} on $M$ if
\begin{equation}\label{4'}
g(\dot {C},\xi )=\eta (\dot {C})=a
\end{equation}
and $a$ is a real constant. The curve $C$ is called a {\it Legendre curve} if $a=0$.

For the sake of brevity, let us use the following denotation
\begin{equation}\label{b}
g(\dot {C},\varphi \dot {C})=b,
\end{equation}
where $b$ is a smooth function on $C$.

\section{Cartan framed slant null curves with respect to the original parameter on 3-dimensional
${\F}_1$-manifolds}\label{sec-3}

\begin{lem}\label{Lemma 3.1}
Let $C$ be a null curve on a 3-dimensional almost contact B-metric manifold $(M,\varphi,\xi,\eta,g)$. Then the triad of vector fields
$\{\dot {C}, \xi, \varphi \dot {C}\}$ is a basis of $T_xM$ at $x\in C$.
\end{lem}
\begin{proof}
At each point $x$ of $C$ we have \eqref{1'} and \eqref{2'}, \ie
$\xi $ is space-like and
$\dot {C}$ is lightlike. Hence $\xi$ and $\dot {C}$
are linearly independent vector fields along $C$. If we assume that $\varphi \dot {C}$ belongs to the plane $\alpha ={\rm span}\{\xi, \dot {C}\}$, we have
$\varphi \dot {C}=u\xi +v\dot {C}$ for some functions $u$ and $v$. Applying $\varphi$ to the both sides of this equality, we obtain $-\dot {C}+\eta (\dot {C})\xi =
uv\xi +v^2\dot {C}$ which implies $v^2=-1$. Thus, $\{\dot {C}, \xi, \varphi \dot {C}\}$ are linearly independent vector fields along $C$ which confirms our assertion.
\end{proof}

\begin{prop}\label{Proposition 3.1}
Let $C$ be a slant null curve with conditions \eqref{4'} and \eqref{b} on
$(M,\varphi,\xi,\eta,g)$, $\dim M=3$.
If ${\bf F}=\{\dot {C}, N, W\}$ is a general Frenet frame on $M$ along $C$ which has the same positive orientation as a basis $\{\dot {C}, \xi, \varphi \dot {C}\}$ at
each $x\in C$, then
\begin{equation}\label{3.1}
W=\alpha \xi +\beta \dot {C}+\gamma \varphi \dot {C} , \qquad
N=\lambda \xi+\mu \dot {C}+\nu \varphi \dot {C},
\end{equation}
where $\beta $ is an arbitrary function and $\alpha, \gamma, \lambda, \mu, \nu $ are the following functions
\begin{equation}\label{3.2}
\begin{split}
\alpha ={-\frac{b}{\sqrt{a^4+b^2}}},\quad
\gamma ={\frac{a}{\sqrt{a^4+b^2}}},
\end{split}
\end{equation}
\begin{equation}\label{3.33}
\begin{split}
\lambda &=\frac{a^3+\beta b\sqrt{a^4+b^2}}
{a^4+b^2},
\\[4pt]
\mu&=-\frac{a^2+\beta ^2\left(a^4+b^2\right)}{2\left(a^4+b^2\right)},
\\[4pt]
\nu &=\frac{b-\beta a
\sqrt{a^4+b^2}}{a^4+b^2}.
\end{split}
\end{equation}
Moreover, the functions $h$ and $k_1$ 
with respect to ${\bf F}$ are given by
\begin{equation}\label{3.4}
\begin{array}{ll}
h=-\lambda g(\dot {C},\nabla _{\dot {C}}\xi )+\frac{\nu }{2}\left[\dot {C}\left(b\right)-F(\dot {C},\dot {C},\dot {C})\right],\\ \\
k_1=\alpha g(\dot {C},\nabla _{\dot {C}}\xi )+\frac{\gamma }{2}\left[\dot {C}\left(b\right)-F(\dot {C},\dot {C},\dot {C})\right].
\end{array}
\end{equation}
\end{prop}
\begin{proof}
According to \lemref{Lemma 3.1}, we have $W=\alpha \xi +\beta \dot {C}+\gamma \varphi \dot {C}$ for some functions $\alpha, \beta, \gamma $. Using 
\eqref{3'}, we obtain the following system of equations for $\alpha, \beta, \gamma $
\begin{equation*}
\begin{array}{l}
a\alpha +\gamma b=0, \\
\alpha ^2+a^2\gamma ^2+2\beta \left(a\alpha +\gamma b\right)=1.
\end{array}
\end{equation*}
The solutions of the system denoting $\varepsilon=\pm 1$ are
\[
\gamma ={\frac{\varepsilon a}{\sqrt{a^4+b^2}}}, \qquad
\alpha =-{\frac{\varepsilon b}{\sqrt{a^4+b^2}}}
\]
and $\beta$ is an arbitrary function.

Let us express $N=\lambda \xi+\mu \dot {C}+\nu \varphi \dot {C}$ for some functions $\lambda, \mu, \nu$.
From \eqref{3'} 
we obtain the following system of equations for $\lambda, \mu, \nu$
\begin{equation*}
\begin{array}{lll}
a\lambda +\nu b=1, \\
\lambda \alpha +\beta +a^2\nu\gamma =0, \\
\lambda^2+a^2\nu^2+2\mu =0,
\end{array}
\end{equation*}
which yields the following solution 
\begin{equation}\label{3.5}
\lambda =\frac{1-\nu b}{a},\qquad 
\nu =-\frac{\alpha +a\beta }{a^3\gamma -\alpha b},\qquad 
\mu =-\frac{1}{2}\left(\lambda ^2+a^2\nu ^2\right).
\end{equation}
The frame ${\bf F}$ is positive oriented if $\gamma ={\frac{a}{\sqrt{a^4+b^2}}}$ holds. This  implies
\[
\alpha ={-\frac{b}{\sqrt{a^4+b^2}}}.
\]
Then we obtain \eqref{3.33} by direct substitutions.
Now, from \eqref{general Frenet eq} and \eqref{3.1} we have
\begin{equation}\label{3.6}
\begin{split}
h&=g\left(\nabla _{\dot {C}}\dot {C},N\right)\\[4pt]
&=\lambda g\left(\nabla _{\dot {C}}\dot {C},\xi \right)+\mu g\left(\nabla _{\dot {C}}\dot {C},\dot {C}\right)+\nu
g\left(\nabla _{\dot {C}}\dot {C},\varphi \dot {C}\right), \\[4pt]
k_1&=g\left(\nabla _{\dot {C}}\dot {C},W\right)\\[4pt]
&=\alpha g\left(\nabla _{\dot {C}}\dot {C},\xi \right)+\beta g\left(\nabla _{\dot {C}}\dot {C},\dot {C}\right)+\gamma
g\left(\nabla _{\dot {C}}\dot {C},\varphi \dot {C}\right).
\end{split}
\end{equation}
Since \eqref{2'} and \eqref{4'} are valid, 
it follows that
\begin{equation}\label{3.7}
g\left(\nabla _{\dot {C}}\dot {C},\dot {C}\right)=0, \qquad 
g\left(\nabla _{\dot {C}}\dot {C},\xi \right)=-g\left(\dot {C},\nabla _{\dot {C}}\xi \right).
\end{equation}
Using the following expressions
\[
\begin{split}
F(\dot {C},\dot {C},\dot {C})=g\left(\nabla _{\dot {C}}\varphi \dot {C},\dot {C}\right)-g\left(\varphi\nabla _{\dot {C}}\dot {C},\dot {C}\right),\\[4pt]
\dot {C}\left(b\right)=g\left(\nabla _{\dot {C}}\dot {C},\varphi \dot {C}\right)+g\left(\dot {C},\nabla _{\dot {C}}\varphi \dot {C}\right),
\end{split}
\]
we obtain
\begin{equation}\label{3.8}
g\left(\nabla _{\dot {C}}\dot {C},\varphi \dot {C}\right)=\frac{1}{2}\left[\dot {C}\left(b\right)-F(\dot {C},\dot {C},\dot {C})\right].
\end{equation}
Substituting \eqref{3.7} and \eqref{3.8} in \eqref{3.6}, we get \eqref{3.4}.
\end{proof}

\begin{cor}\label{Corollary}
Let the assumptions of \propref{Proposition 3.1} are satisfied and $(M,\varphi,\allowbreak{}\xi,\allowbreak{}\eta,g)$ be in ${\F}_1$.
Then we have
\begin{equation}\label{3.10}
\begin{split}
h=\frac{\nu }{2}\left[\dot {C}\left(b\right)+a^2\theta (\dot {C})-b\theta (\varphi \dot {C})\right], \\
k_1=\frac{\gamma }{2}\left[\dot {C}\left(b\right)+a^2\theta (\dot {C})-b\theta (\varphi \dot {C})\right].
\end{split}
\end{equation}
\end{cor}
\begin{proof}
By virtue of \eqref{2.2} we get $\theta (\xi )=0$ and hence $\theta (\varphi ^2\dot {C})=-\theta (\dot {C})$. Then
\begin{equation}\label{3.12}
F(\dot {C},\dot {C},\dot {C})=b\theta (\varphi \dot {C})-a^2\theta (\dot {C}).
\end{equation}
Substituting \eqref{3.12} in \eqref{3.4}, we obtain \eqref{3.10}.
\end{proof}

\begin{prop}\label{Proposition 3.2}
A slant null curve $C$ on a 3-dimensional ${\F}_1$-manifold is geodesic if and only if the following equality holds
\[
\dot {C}\left(b\right)=b\theta (\varphi \dot {C})-a^2\theta (\dot {C}).
\]
\end{prop}
\begin{proof}
As it is known (\cite{D-J}), a null curve is geodesic if and only $k_1$ vanishes. Taking into account \eqref{3.10} and $\gamma \neq 0$, we establish the truthfulness of the statement.
\end{proof}

Immediately we obtain the following

\begin{cor}\label{cor 3.2}
A slant null curve $C$ on a 3-dimensional ${\F}_0$-manifold is geodesic if and only if $b$ is a constant.
\end{cor}

Bearing in mind the statement for the general Frenet frames in \propref{Proposition 3.1}, we have the following
\begin{thm}\label{Theorem 3.1}
Let $C(t)$ be a non-geodesic slant null curve on a 3-dimensional ${\F}_1$-manifold $(M,\varphi,\xi,\eta,g)$. Then there exists a unique Frenet frame
${\bf F}_1=\{\dot {C}, N_1, W_1\}$ for which the original parameter $t$ is distinguished
and
\begin{equation}\label{3.13}
W_1=\alpha \xi-\frac{\alpha }{a}\dot {C}+\gamma \varphi \dot {C} , \qquad
N_1=\frac{1}{a}\xi -\frac{1}{2a^2}\dot {C} ,
\end{equation}
where $\alpha $ and $\gamma $ are given by \eqref{3.2}.
\end{thm}
\begin{proof}
The original parameter $t$ of a null curve $C(t)$ is distinguished if $h(t)$ vanishes.
Since $C(t)$ is  non-geodesic, by using \propref{Proposition 3.2} and \corref{cor 3.2}, we obtain for an ${\F}_1$-manifold  that $h=0$ if and only if $\nu =0$.
The second equality of \eqref{3.5} implies that $\nu =0$ if and only if $\beta =-\frac{\alpha }{a}$. Substituting $\beta =-
\frac{\alpha }{a}$ in \eqref{3.33}, we get $\lambda =\frac{1}{a}$ and $\mu =-\frac{1}{2a^2}$. The vector fields
$W_1$ and $N_1$ in \eqref{3.13} are obtained from \eqref{3.1} by $\beta =-\frac{\alpha }{a}$, $\lambda =\frac{1}{a}$, $\mu =-\frac{1}{2a^2}$ and $\nu =0$. Hence the Frenet frame ${\bf F}_1=\{\dot {C}, N_1, W_1\}$ is the unique frame for which the original parameter $t$ is distinguished.
\end{proof}
\begin{thm}\label{Theorem 3.2}
Let $C(t)$
be a slant null curve on a 3-dimensional ${\F}_1$-manifold $(M,\varphi,\xi,\eta,g)$ and $b$ satisfies the following
ordinary differential equation
\begin{equation}\label{3.15}
\dot {C}\left(b\right)=b\theta (\varphi \dot {C})-a^2\theta (\dot {C})+\frac{2}{a}\sqrt{a^4+b^2}.
\end{equation}
Then $(C(t), {\bf F}_1)$ is a non-geodesic Cartan framed slant null curve, where ${\bf F}_1=\{\dot {C}, N_1, W_1\}$ is the unique Frenet frame of $C(t)$ from
\thmref{Theorem 3.1}. Moreover, the torsion function is $\tau =-\frac{1}{2a^2}$.\\
In particular, if $(M,\varphi,\xi,\eta,g)$ is an ${\F}_0$-manifold then $b$ is given by
\begin{equation}\label{3.14}
b={\frac{1}{2}\left[\exp\left({\frac{2(t+u)}{a}}\right)-a^4\exp\left({-\frac{2(t+u)}{a}}\right)\right]},
\end{equation}
where $u$ is an arbitrary real constant.
\end{thm}
\begin{proof}
Taking into account \propref{Proposition 3.2} and \eqref{3.15}, we conclude that  $C(t)$ is non-geodesic. Then it follows from \thmref{Theorem 3.1} that there exists
a unique Frenet frame ${\bf F}_1=\{\dot {C}, N_1, W_1\}$ for which the original parameter $t$ is distinguished.
Substituting the expression for $\gamma $ from \eqref{3.2} and \eqref{3.15} in \eqref{3.10}, we obtain $k_1(t)=1$. Hence  ${\bf F}_1$ is a Cartan Frenet frame with respect to $t$.
The Cartan Frenet equations \eqref{Cartan Frenet eq} with respect to ${\bf F}_1$ imply $\tau =g(\nabla _{\dot {C}}N_1,W_1)$. Using \eqref{3.13} and \eqref{Cartan Frenet eq}, we obtain the equality
$\nabla _{\dot {C}}N_1=-\frac{1}{2a^2}W_1$. Hence $\tau =-\frac{1}{2a^2}$ holds.

In the particular case, when $(M,\varphi,\xi,\eta,g)$ is an ${\F}_0$-manifold, $b$ satisfies the following ordinary differential equation
\begin{equation*}\label{3.16}
\frac{a\,{\rm d}b}{2\sqrt{a^4+b^2}}={\rm d}t ,
\end{equation*}
which is obtained from \eqref{3.15} by $\theta =0$.
Integrating the latter equation, we get
\[
\frac{a}{2}{\rm ln}\left(b+\sqrt{a^4+b^2}\right)=t+u, \qquad u\in {\R}.
\]
The last equality implies \eqref{3.14}.
\end{proof}

\section{Examples of non-geodesic Cartan framed slant null curves with respect to the original parameter on 3-dimensional ${\F}_0$- and ${\F}_1$-manifolds}\label{sec-4}

\subsection{A Minkowski space equipped with a cosymplectic B-metric structure}

An almost contact structure $(\varphi,\xi,\eta )$ and a B-metric $g$ on ${\R}^{2n+1}$ are defined in \cite[Example 1., p. 270]{GaMGri} and it is shown that
$({\R}^{2n+1},\varphi,\xi,\eta,g)$ is an ${\F}_0$-manifold. Now, we consider such an ${\F}_0$-manifold in dimension 3, where the structure $(\varphi,\xi,\eta,g)$ is defined by
\[
\xi =\frac{\partial}{\partial x_3}, \, \, \, \eta ={\rm d}x_3, \, \, \, \varphi \left(\frac{\partial}{\partial x_1}\right)=\frac{\partial}{\partial x_2}, \, \, \,
\varphi \left(\frac{\partial}{\partial x_2}\right)=-\frac{\partial}{\partial x_1}, \, \, \,
\varphi \left(\frac{\partial}{\partial x_3}\right)=0,
\]
\[
g(x,x)=-x_1^2+x_2^2+x_3^2
\]
for $x={x_1\frac{\partial}{\partial x_1}+x_2\frac{\partial}{\partial x_2}+x_3\frac{\partial}{\partial x_3}}$. It is clear that $({\R}^3,\varphi,\xi,\eta,g)$
is a Minkowski space, which is equipped with an almost contact B-metric structure.

Let $C_1(t)=(x_1(t), x_2(t), x_3(t))$, $t\in I$, be a slant null curve on $({\R}^3,\varphi,\xi,\eta,g)$. Then the conditions \eqref{2'} and \eqref{4'}
imply
\begin{equation}\label{4.0}
-\dot {x}_1^2+\dot {x}_2^2+\dot {x}_3^2=0
\end{equation}
and
\begin{equation}\label{4.111}
\dot {x}_3=a ,
\end{equation}
respectively.
Taking into account that $\varphi \dot {C}_1=(-\dot {x}_2, \dot {x}_1, 0)$, we get
\[
b=g(\dot {C}_1,\varphi \dot {C}_1)=2\dot {x}_1\dot {x}_2 .
\]
According to \thmref{Theorem 3.2}, $({C}_1(t), {\bf F}_1)$ is a non-geodesic Cartan framed slant null curve on $({\R}^3,\varphi,\xi,\eta,g)$ if the
condition \eqref{3.14} holds, \ie
\begin{equation}\label{4.2}
2\dot {x}_1\dot {x}_2={\frac{1}{2}\left[\exp\left(\frac{2(t+u)}{a}\right)-a^4\exp\left(-\frac{2(t+u)}{a}\right)\right]}.
\end{equation}
From \eqref{4.0} we obtain ${\dot {x}_1=\pm\sqrt{\dot {x}_2^2+a^2}}$ which implies $\dot {x}_1\neq 0$ for all $t\in I$. If we take ${\dot {x}_1=\sqrt{\dot {x}_2^2+a^2}}$, then
\eqref{4.2} becomes
\begin{equation*}\label{4.3}
{2\dot {x}_2\sqrt{\dot {x}_2^2+a^2}=\frac{1}{2}\left[\exp\left(\frac{2(t+u)}{a}\right)-a^4\exp\left(-\frac{2(t+u)}{a}\right)\right]}.
\end{equation*}
We solve the latter equation and obtain
\begin{equation}\label{4.4}
{\dot {x}_2=\frac{1}{2}\left[\exp\left(\frac{t+u}{a}\right)-a^2\exp\left(-\frac{t+u}{a}\right)\right]}.
\end{equation}
Substituting \eqref{4.4} in \eqref{4.2}, we find
\begin{equation}\label{4.5}
{\dot {x}_1=\frac{1}{2}\left[\exp\left(\frac{t+u}{a}\right)+a^2\exp\left(-\frac{t+u}{a}\right)\right]}.
\end{equation}
Integrating \eqref{4.111}, \eqref{4.4} and \eqref{4.5}, we get
\begin{equation*}
\begin{split}
C_1(t)&=\left(\frac{a}{2}\left[\exp\left(\frac{t+u}{a}\right)-a^2\exp\left(-\frac{t+u}{a}\right)\right]+c_1, \right. \\
&\phantom{=\left(\right.\,}
\left.\frac{a}{2}\left[\exp\left(\frac{t+u}{a}\right)+a^2\exp\left(-\frac{t+u}{a}\right)\right]+c_2, \, \, \, at+c_3\right),
\end{split}
\end{equation*}
where $c_1, \,  c_2, \, c_3 \in {\R}$.

Analogously, in the case ${\dot {x}_1=-\sqrt{\dot {x}_2^2+a^2}}$ we have
\begin{equation*}
\begin{split}
C_2(t)&=\left(-\frac{a}{2}\left[\exp\left(\frac{t+u}{a}\right)-a^2\exp\left(-\frac{t+u}{a}\right)\right]+c_4, \right. \\
&\phantom{=\left(\right.\,}
\left.-\frac{a}{2}\left[\exp\left(\frac{t+u}{a}\right)+a^2\exp\left(-\frac{t+u}{a}\right)\right]+c_5, \, \, \, at+c_6\right),
\end{split}
\end{equation*}
where $c_4, \,  c_5, \, c_6 \in {\R}$.

Thus, we state
\begin{thm}\label{4.1}
The unique non-geodesic Cartan framed slant null curves with respect to the original parameter $t$ on $({\R}^3,\varphi,\xi,\eta,g)$
are  $(C_i(t), {\bf F}_i)$, where ${\bf F}_i=\allowbreak{}\{\dot {C}_i, \allowbreak{}N_i, \allowbreak{}W_i\}$ \, $(i=1,2)$ and
\begin{equation*}\label{4.6}
\begin{split}
N_i&=\frac{1}{a}\xi -\frac{1}{2a^2}\dot {C}_i, \\
W_i&=-\rho
\left(\xi -\frac{1}{a}\dot {C}_i-
{\left.\frac{2a} {\exp\left(\frac{2(t+u)}{a}\right)-a^4\exp\left(-\frac{2(t+u)}{a}\right)}\varphi \dot {C}_i\right)},\right.
\end{split}
\end{equation*}
where
\[
\rho=\frac{\exp\left(\frac{2(t+u)}{a}\right)
-a^4\exp\left(-\frac{2(t+u)}{a}\right)}
{\exp\left(\frac{2(t+u)}{a}\right)+a^4\exp\left(-\frac{2(t+u)}{a}\right)}.
\]
\end{thm}

\subsection{A non-geodesic Cartan framed slant null curve on an $\F_1$-manifold constructed on a Lie group}

Let $L$ be a 3-dimensional real connected Lie group and $\mathfrak {l}$ its corresponding Lie algebra. If $\{E_0, E_1, E_2\}$ is a basis of left invariant
vector fields of $\mathfrak {l}$ then $L$ is equipped with an almost contact structure $(\varphi, \xi, \eta )$ and a left invariant B-metric $g$ in \cite{HM}
as follows:
\begin{equation}\label{4.7}
\begin{array}{llll}
\varphi E_0=0, \quad \varphi E_1=E_2, \quad \varphi E_2=-E_1, \quad \xi =E_0, \\
\eta (E_0)=1, \quad \eta (E_1)=\eta (E_2)=0, \\
g(E_0,E_0)=g(E_1,E_1)=-g(E_2,E_2)=1, \\
g(E_0,E_1)=g(E_0,E_2)=g(E_1,E_2)=0.
\end{array}
\end{equation}
Let $(L,\varphi,\xi,\eta,g)$ be a 3-dimensional almost contact B-metric manifold belonging to the class ${\F}_1$.
It is proved in \cite[Theorem 1.1]{HM} that the corresponding Lie algebra $\mathfrak {g}_1$ of $L$ is determined by the following commutators:
\begin{equation*}
[E_0,E_1]=[E_0,E_2]=0 \qquad [E_1,E_2]=\alpha E_1+\beta E_2,
\end{equation*}
where $\alpha, \, \beta $ are arbitrary real parameters such that $(\alpha, \beta )\neq (0, 0)$ and
\begin{equation*}\label{4.8}
\alpha =\frac{1}{2}\theta _1, \qquad \beta =\frac{1}{2}\theta _2.
\end{equation*}
In the latter equalities, by $\theta _1$ and $\theta _2$ are denoted $\theta (E_1)$ and $\theta (E_2)$, respectively. We note that $\theta _0=\theta (\xi )=0$.

Let $G$ be the compact simply connected Lie group with the same Lie algebra as $L$ as well as $G$ is isomorphic to $L$. Consider the curve $C(t)=\exp(tX)$ on $G$, where
$t\in {\R}$ and $X\in \mathfrak {g}_1$. Hence the tangent vector to $C(t)$ at the identity element $e$ of $G$ is $\dot {C}(0)=X$. Let $X$ satisfies the following
conditions
\begin{equation}\label{4.9}
g(X,X)=0, \qquad \eta (X)=a, \, \, a\in {\R}\setminus \{0\}.
\end{equation}
Since $g$ is left invariant, from \eqref{4.9} it follows that $g(\dot {C}(t),\dot {C}(t))=0$ and  $\eta (\dot {C}(t))=a$ for all $t\in {\R}$. This means that $C(t)$ is a slant null curve on $G$. Also we have $b=g(\dot {C}(t),\varphi \dot {C}(t))=g(X,\varphi X)$ for all $t\in {\R}$ and therefore $b$ is a constant. Taking into account
\eqref{4.9}, the coordinates $(p,q,r)$ of $X$ with respect to the basis $\{E_0, E_1, E_2\}$ satisfy the following equalities
\begin{equation}\label{4.10}
p^2+q^2-r^2=0, \qquad p=a.
\end{equation}
From \eqref{4.10} we obtain $r=\pm \sqrt{a^2+q^2}$. Let us suppose that $r=\sqrt{a^2+q^2}$. Having in mind \eqref{4.7}, we get
$\varphi X=(0,-\sqrt{a^2+q^2},q)$. Then we have
\begin{equation}\label{4.11}
b=g(X,\varphi X)=-2q\sqrt{a^2+q^2}.
\end{equation}
Using \eqref{4.8}, we obtain
\begin{equation}\label{4.12}
\theta (\varphi X)=2\beta q-2\alpha \sqrt{a^2+q^2}, \quad \theta (X)=2\alpha q+2\beta \sqrt{a^2+q^2}.
\end{equation}
Now, according to \thmref{Theorem 3.2}, $(C(t), {\bf F}_1)$ is a non-geodesic Cartan framed slant null curve on $G$ if  \eqref{3.15} holds. By using \eqref{4.11}
and \eqref{4.12}, the equation \eqref{3.15} becomes
\begin{equation*}\label{4.14}
(a^2+2q^2)\left(a\beta \sqrt{a^2+q^2}-\alpha aq-1\right)=0,
\end{equation*}
which is equivalent to
\begin{equation}\label{4.15}
a\beta \sqrt{a^2+q^2}-\alpha aq-1=0
\end{equation}
since $a^2+2q^2\neq 0$.

Consider an ${\F}_1$-manifold $(L_1,\varphi,\xi,\eta,g)$ such that $\alpha =\beta \neq 0$ and $b\neq 0$. For this manifold  \eqref{4.15} becomes
\begin{equation}\label{4.16}
\alpha a\sqrt{a^2+q^2}=1+\alpha aq.
\end{equation}
If $\alpha a>0$ then \eqref{4.16} has a unique solution $q=\frac{\alpha ^2a^4-1}{2\alpha a}$. Then we get
$r=\frac{\alpha ^2a^4+1}{2\alpha a}$. The condition $b\neq 0$ implies $q\neq 0$ which means that
$ \alpha \neq \pm \frac{1}{a^2}$.

Thus, for ${\bf F}_1=\{X, N_1, W_1\}\in T_eG$ and  $\varphi X$ we obtain
\begin{equation}\label{4.17}
X=\left(a,\, \frac{\alpha ^2a^4-1}{2\alpha a},\, \frac{\alpha ^2a^4+1}{2\alpha a}\right),
\end{equation}
\begin{equation}\label{4.177}
\begin{split}
N_1&=\frac{1}{a}\xi -\frac{1}{2a^2}X, \\
W_1&=\frac{\alpha ^4a^8-1}{\alpha ^4a^8+1}\left(\xi -\frac{1}{a}X+\frac{2\alpha ^2a^3}{\alpha ^4a^8-1}\varphi X\right),
\end{split}
\end{equation}
\begin{equation}\label{4.1777}
\varphi X=\left(0, -\frac{\alpha ^2a^4+1}{2\alpha a}, \frac{\alpha ^2a^4-1}{2\alpha a}\right),
\end{equation}
where $\alpha a>0$ and $ \alpha \neq \pm \frac{1}{a^2}$.

Further, we find the matrix representation of $C(t)$ and ${\bf F}_1$. In \cite[Theorem 2.1]{HM} it is found explicitly the 3-dimensional matrix representation $\pi $
of the Lie algebra $\mathfrak {g}_1$. It is well known that $\pi $ is the following Lie algebra homomorphism \\
$\pi : \mathfrak {g}_1\longrightarrow \Hom(V)$ such that $Y\longrightarrow \pi (Y)$, where $V$ is a 3-dimensional real vector space. Notice that the linear operators $\pi (Y)\in \Hom(V)$ do not need to be invertible. Their matrices $A$ are called briefly matrix representation of
$\mathfrak {g}_1$. It is proved that
\[
\pi (E_0)=\left(\begin{array}{lll}
0 & 0 & 0 \cr
0 & 0 & 0 \cr
0 & 0 & 0
\end{array}\right), \quad
\pi (E_1)=\left(\begin{array}{lcc}
0 & 0 & 0 \cr
0 & 0 & 0 \cr
0 & -\alpha & -\beta
\end{array}\right),
\]
\[
\pi (E_2)=\left(\begin{array}{lll}
0 & 0 & 0 \cr
0 & \alpha  & \beta \cr
0 & 0 & 0
\end{array}\right).
\]
Now, taking into account that $\alpha =\beta $ and using \eqref{4.17}, \eqref{4.1777}, we have
\begin{equation}\label{4.18}
\pi (X)=
\left(\begin{array}{lcc}
0 & 0 & 0 \cr \cr
0 & \frac{\alpha ^2a^4+1}{2a} & \frac{\alpha ^2a^4+1}{2a} \cr \cr
0 & \frac{1-\alpha ^2a^4}{2a} & \frac{1-\alpha ^2a^4}{2a}
\end{array}\right),
\end{equation}
\begin{equation}\label{4.188}
\pi (\varphi X)=
\left(\begin{array}{lcc}
0 & 0 & 0 \cr \cr
0 & \frac{\alpha ^2a^4-1}{2a} & \frac{\alpha ^2a^4-1}{2a} \cr \cr
0 & \frac{\alpha ^2a^4+1}{2a} & \frac{\alpha ^2a^4+1}{2a}
\end{array}\right).
\end{equation}
Let $\Pi: G\longrightarrow \End(V)$ be the matrix representation of the simply connected Lie group $G$. By $\End(V)$ are denoted the invertible linear operators of $V$. The representations $\Pi $ and $\pi $ are related as follows:
\[
\Pi (\exp(Y))=\exp(\pi (Y))
\]
for  all $Y\in \mathfrak {g}_1$.
Since $tX\in \mathfrak {g}_1$, $t\in {\R}$, we get
\begin{equation}\label{4.19}
\Pi (C(t))=\Pi (\exp(tX))=\exp(\pi (tX)).
\end{equation}
In \cite{HM}, the group of the matrices of the endomorphisms $\Pi (b), \, b\in G$, is denoted by $G_1$ and it is called \emph{the matrix Lie group representation}.
According to \cite[Theorem 2.1]{HM}, we have
\begin{equation}\label{4.20}
G_1=\left\{\exp(A)=E+\left(\frac{\exp(\tr A)-1}{\tr A}\right)A\right\},
\end{equation}
where $\tr A\neq 0$. Using \eqref{4.18}, we find $\tr\left( \pi (tX)\right)=\frac{t}{a}$ and taking into account \eqref{4.20}, we get
\begin{equation}\label{4.21}
\exp(\pi (tX))=E+a\left(\exp\left(\frac{t}{a}\right)-1\right)\pi (X).
\end{equation}
Substituting \eqref{4.21} in \eqref{4.19} and taking into account \eqref{4.18}, we obtain the following matrix representation  of $C(t)$
\begin{equation*}
\Pi (C(t))=
\left(\begin{array}{lcc}
1 & 0 & 0 \cr \cr
0 & 1+\exp\left(\frac{t}{a}-1\right)\frac{\alpha ^2a^4+1}{2} & \exp\left(\frac{t}{a}-1\right)\frac{\alpha ^2a^4+1}{2} \cr \cr
0 & \exp\left(\frac{t}{a}-1\right)\frac{1-\alpha ^2a^4}{2} &
1+\exp\left(\frac{t}{a}-1\right)\frac{1-\alpha ^2a^4}{2}
\end{array}\right).
\end{equation*}
Finally, using \eqref{4.177}, \eqref{4.18} and \eqref{4.188}, we obtain the  matrix representations of $N_1$ and $W_1$ as follows:
\begin{equation*}
\pi (N_1)=
\left(\begin{array}{lcc}
0 & 0 & 0 \cr \cr
0 & -\frac{\alpha ^2a^4+1}{4a^3} & -\frac{\alpha ^2a^4+1}{4a^3} \cr \cr
0 & \frac{\alpha ^2a^4-1}{4a^3} & \frac{\alpha ^2a^4-1}{4a^3}
\end{array}\right),
\end{equation*}
\begin{equation*}
\pi (W_1)=
\left(\begin{array}{lcc}
0 & 0 & 0 \cr \cr
0 & \frac{1-\alpha ^2a^4}{2a^2} & \frac{1-\alpha ^2a^4}{2a^2} \cr \cr
0 & \frac{1+\alpha ^2a^4}{2a^2} & \frac{1+\alpha ^2a^4}{2a^2}
\end{array}\right).
\end{equation*}


\end{document}